\documentclass{svjour2}                    
\usepackage{graphicx}
\usepackage[T1]{fontenc}
\usepackage{amsmath}
\usepackage{amssymb}

\def\E {\mathbb{E}}

\def\R {\mathbb{R}}
\def\N {\mathbb{N}}
\def\I {\mathbb{I}}

\def\CG {\mathcal{G}}
\def\CP {\mathcal{P}}

\def\CF {\mathcal{F}}
\def\CM {\mathcal{M}}
\def\CL {\mathcal{L}}
\def\CB {\mathcal{B}}

\def\CI {\mathcal{I}}

\def\vphi {\varphi}
\def\veps {\varepsilon}
\def\lm {\lambda}
\def\o {\omega}
\def\O {\Omega}

\def\cq {\hat{q}}

\def\co {\hat{\omega}}
\def\cO {\widehat{\Omega}}
\def\cF {\widehat{\CF}}
\def\cP {\widehat{\CP}}
\def\cW {\widehat{W}}
\def\cY {\widehat{Y}}
\def\cZ {\widehat{Z}}
\def\cM {\widehat{M}}
\def\cz {\hat{\zeta}}
\def\cw {\hat{w}}
\def\ci {\hat{i}}

\def\tte {\tilde{\theta}}
\def\ti {\tilde{i}}
\def\tnu {\tilde{\nu}}
\def\tbe {\tilde{\beta}}
\def\tlm {\tilde{\lambda}}
\def\tu {\tilde{u}}
\def\tLm {\tilde{\Lambda}}

\def\na {\frac{1}{n^{\frac{1}{\alpha}}}}

\def\cnu {\hat{\nu}}
\def\cbe {\hat{\beta}}
\def\cte {\hat{\theta}}

\def\cnun {\hat{\nu}^{(n)}}
\def\cten {\hat{\theta}^{(n)}}
\def\cin {\hat{i}^{(n)}}
\def\cben {\hat{\beta}^{(n)}}
\def\clmn {\hat{\lambda}^{(n)}}
\def\cLmn {\widehat{\Lambda}^{(n)}}
\def\cun {\hat{u}^{(n)}}

\def\ctei {\hat{\theta}^{(\infty)}}
\def\cii {\hat{i}^{(\infty)}}
\def\cbei {\hat{\beta}^{(\infty)}}
\def\cLmi {\widehat{\Lambda}^{(\infty)}}

\def\mbar {\overline{m}}

\begin{document}

\title{Complete characterization of convergence to equilibrium for an inelastic Kac model}



\author{Ester Gabetta         \and
        Eugenio Regazzini 
}


\institute{E. Gabetta \at
              Dipartimento di Matematica, Università degli Studi di Pavia, 27100 Pavia, Italy \\
              Tel.: +39.0382.985650\\
              Fax: +39.0382.985602\\
              \email{ester.gabetta@unipv.it}           
           \and
           E. Regazzini \at
              Dipartimento di Matematica, Università degli Studi di Pavia, 27100 Pavia, Italy \\
              Tel.: +39.0382.985622\\
              Fax: +39.0382.985602\\
              \email{eugenio.regazzini@unipv.it}\\
              Also affiliated to CNR-IMATI (Milano)
}
\date{Received: date / Accepted: date}

\maketitle

\begin{abstract}
Pulvirenti and Toscani introduced an equation which extends the Kac caricature of a Maxwellian gas to inelastic particles. We show that the probability distribution, solution of the relative Cauchy problem, converges weakly to a probability distribution if and only if the symmetrized initial distribution belongs to the standard domain of attraction of a symmetric stable law, whose index $\alpha$ is determined by the so-called degree of inelasticity, $p>0$, of the particles: $\alpha=\frac{2}{1+p}$. This result is then used: (1) To state that the class of all stationary solutions coincides with that of all symmetric stable laws with index $\alpha$. (2) To determine the solution of a well-known stochastic functional equation in the absence of extra-conditions usually adopted.

\end{abstract}

\section{Introduction and formulation of the new results}
\label{sec:1}

This work deals with a model introduced in \cite{PulvirentiToscani} to provide an extension to inelastic particles of the well-known \textit{Kac caricature} of a Maxwellian gas. It consists of an equation which describes the evolution of the probability distribution (p.d., for short) $\mu(\cdot,t)$ of the velocity of a particle on the real line. The model can be formulated, in terms of the Fourier-Stieltjes transform $\varphi(\cdot,t)$ of $\mu(\cdot,t)$, as follows

\begin{equation}\label{eqPhi}
\frac{\partial}{\partial t}\varphi(\xi,t)=\frac{1}{2\pi}\int_{0}^{2\pi}\varphi(\xi c_p(\theta),t)\varphi(\xi s_p(\theta),t)d\theta-\varphi(\xi,t)\quad (\xi\in\R,t>0)
\end{equation}
where $p\geq0$ is a parameter, and

\[
c_p(\theta):=\cos\theta|\cos\theta|^p, \quad s_p(\theta):=\sin\theta|\sin\theta|^p\qquad (\theta\in (0,2\pi]).
\]
The parameter $p$ can be interpreted as \textit{degree of inelasticity}. The case of perfect elasticity corresponds to $p=0$, when $(\ref{eqPhi})$ coincides with the Kac equation. Motivations of a physical nature for the analysis of dissipative systems ($p>0$) can be found in Section 1 of \cite{PulvirentiToscani} and in some of the references quoted therein. See also reviews in \cite{Villani2002} and \cite{Villani2006}. Recently, one-dimensional extensions of $(\ref{eqPhi})$ have been proposed and studied in \cite{BassettiLadelli}, \cite{BassettiLadelliMatthes}, \cite{BobylevCercignaniGamba}, and have been reviewed in \cite{BassettiGabetta}. According to these extensions, the random vector $(c_p(\tilde{\theta}), s_p(\tilde{\theta}))$, with $\tilde{\theta}$ uniformly distributed on $(0,2\pi]$, is replaced by some more general random vector $(L,R)$ whose p.d. $m$ satisfies $\int_{\R^2}(|x|^a+|y|^a)m(dxdy)=1$ for some $a$ in $(0,2]$.

Turning back to $(\ref{eqPhi})$, it is a well-known fact that the Cauchy problem obtained by combining $(\ref{eqPhi})$ with the initial condition

\begin{equation}\label{condIniz}
\varphi(\xi,0)=\varphi_0(\xi) \qquad (\xi\in\R),
\end{equation}
where $\varphi_0$ denotes the Fourier-Stieltjes transform of an initial p.d. $\mu_0$, admits one and only one solution in the class of all one-dimensional \textit{characteristic functions} (c.f.). Recall that, given any p.d. $m$ on the Borel class of $\R^d$, $\CB(\R^d)$, the corresponding c.f. is defined by $\hat{m}(\xi):=\int_{\R^d}e^{i(\xi,x)}m(dx)$ for every $\xi$ in $\R^d$, which is the same as the notion of Fourier-Stieltjes transform of $m$, adopted in the present paper. According to the tradition of the kinetic models, the long-time behaviour of the solution of $(\ref{eqPhi})$ has been taken into consideration with a view to the following problems:

\begin{itemize}
\item[(I)] To find all the stationary solutions (equilibria).
\item[(II)] To provide conditions on $\mu_0$ in order that $\mu(\cdot,t)$ converge weakly to some probabilistic distribution $\mu_\infty$, as $t$ goes to infinity.
\item[(III)] To evaluate both the rate of convergence and the error in approximating for fixed $t$.
\end{itemize}

As to (I), it has been proved, in \cite{PulvirentiToscani}, that symmetric stable p.d.'s $g_\alpha$, with c.f.

\begin{equation}\label{cfStable}
\hat{g}_\alpha(\xi)=e^{-a_0|\xi|^\alpha}\qquad (\xi\in\R)
\end{equation}
for any $a_0$ in $\R^+:=[0,+\infty)$ and

\begin{equation}\label{alpha}
\alpha:=\frac{2}{1+p}
\end{equation}\newline
are stationary solutions for $(\ref{eqPhi})$. As far as problem (II) is concerned, recall that a sequence of p.d.'s $m_n$ on the Borel class $\CB(M)$ of a metric space $M$ is said to be weakly convergent to a p.d. $m$ (in symbols: $m_n\Rightarrow m$) if $\int_{M}h dm_n\rightarrow \int_{M}h dm$ as $n\rightarrow +\infty$, for every bounded and continuous function $h\colon M\rightarrow\R$. Now, let $F_0$ indicate the p.d. function associated with the initial datum $\mu_0$, $-$ $F_0(x):=\mu_0((-\infty,x])$ for every $x$ in $\R$ $-$ and let $F^*_0$ denote the symmetrized p.d. function determinated by

\begin{equation}\label{simm}
F^*_0(x):=\frac{1}{2}[F_0(x)+1-F_0(-x)]
\end{equation}
at each $x$ such that $(-x)$ is a continuity point for $F_0$. In this notation a partial solution to problem (II) has been proved in \cite{BassettiLadelliRegazzini}: \textit{If $p>0$ and}

\begin{equation}\label{cond}
\lim_{x\rightarrow+\infty}x^\alpha[1-F^*_0(x)]=c_0
\end{equation}
\textit{for some $c_0$ in $\R^+$, then $\mu(\cdot,t)\Rightarrow g_\alpha$ as $t\rightarrow+\infty$, with}

\begin{equation}\label{a_0}
a_0=2 c_0\lim_{T\rightarrow+\infty}\int_{0}^{T}\frac{\sin x}{x^\alpha}dx.
\end{equation}\newline
Both in \cite{BassettiLadelliRegazzini} and in \cite{PulvirentiToscani}, answers to problem (III) have been given with respect to certain weak metrics in the set of all p.d.'s on $\CB(\R)$.

As to the present work, its main goal is that of checking whether $(\ref{cond})$, besides being sufficient, is also necessary in order that the solution to $(\ref{eqPhi})$-$(\ref{condIniz})$ may converge weakly. When $p=0$ (Kac equation), complete solutions to (I)-(II) can be found in \cite{CarlenGabettaRegazzini2008} and \cite{GabettaRegazzini2008}: \textit{The class of all stationary solutions of the Kac equation is the same as that of all weak limits and coincides with the set of all Gaussian laws with $0$ mean; moreover, convergence to equilibrium happens if and only if the variance of the initial datum is finite.} The most important result in the present paper is of the same nature and provides a complete answer to (II) in the case of $p>0$ ($0<\alpha<2$). A partial response is given by Theorem 2.1 in \cite{BassettiLadelliRegazzini}: \textit{If $\mu(\cdot,t)$ converges weakly, then}
\[
\lim_{\xi\rightarrow+\infty}\inf_{x\geq\xi}x^\alpha[1-F^*_0(x)]<+\infty.
\]
This statement is now completed and improved by

\begin{theorem}\label{teorema1}
Let $0<\alpha<2$. Then, in order that the solution $\mu(\cdot,t)$ of $(\ref{eqPhi})$-$(\ref{condIniz})$ converge weakly to a p.d. $\mu_\infty$ on $\CB(\R)$ it is necessary and sufficient that $(\ref{cond})$ holds true. If this is so, $\hat{\mu}_\infty=\hat{g}_\alpha$ with $a_0$ given by $(\ref{a_0})$. So, the limit $\mu_\infty$ degenerates \emph{(}in the sense that $\mu_\infty$ is the point mass $\delta_{x_0}$ for some $x_0$\emph{)} if and only if $a_0=0$ and, therefore, $\mu_\infty=\delta_0$.
\end{theorem}
The proof is deferred to Section 3. It is based on a probabilistic construction $-$ explained in Section 2 $-$ which permits the application of methods of current usage to prove the \textit{central limit theorem}. In fact, also the conclusions parallel the answer to the central limit problem for independent and identically distributed (i.i.d. for short) summands, in the sense that, also in the kinetic problem, the limiting laws are stable, and ``good'' initial data belong to the \textit{standard domain of attraction} of stable laws. See Chapter 2 of \cite{IbragimovLinnik} and Chapter VII of \cite{Loève}.

The following straightforward corollary of the above theorem provides a complete answer to problem (I).

\begin{corollary}\label{corollario1}
For each $\alpha$ in $(0,2]$, the class of all stationary solutions of $(\ref{eqPhi})$-$(\ref{condIniz})$ is the same as that of all symmetric stable laws $g_\alpha$ as $a_0$ varies in $\R^+$.
\end{corollary}
In accordance with definition, stationary solutions of $(\ref{eqPhi})$ are c.f.'s $\varphi(\cdot,t)$ such that $\frac{\partial}{\partial t}\varphi(\cdot,t)\equiv 0$ or, equivalently, solutions of the equation

\begin{equation}\label{staz}
\frac{1}{2\pi}\int_{0}^{2\pi}\varphi(\xi c_p(\theta))\varphi(\xi s_p(\theta))d\theta=\varphi(\xi)\qquad(\xi\in\R)
\end{equation}
where $\varphi$ is an unknown c.f.. There is a flourishing literature on stochastic functional equations of the type of $(\ref{staz})$, motivated by interesting problems arising in probability and stochastics. See the recent paper \cite{AlsmeyerMeiners} and references therein. It is worth recalling that studies on equations which generalize $(\ref{staz})$ have been utilized, in \cite{BassettiLadelliMatthes} for example, to characterize limits of solutions of generalized dissipative kinetic models. Here, thanks to the complete solution to (I)-(II) given by Theorem $\ref{teorema1}$, we are in a position to exchange roles, in the sense that we utilize Theorem $\ref{teorema1}$ to characterize the complete solution of $(\ref{staz})$. A suggestion to proceed in this way has been given us by Federico Bassetti.

\begin{corollary}\label{corollario2}
The class of all solutions of $(\ref{staz})$, within the class of all c.f.'s on the real line, consists of all c.f.'s $(\ref{cfStable})$ with $a_0\geq0$, for each $\alpha$ in $(0,2]$.
\end{corollary}

\begin{proof}
It is immediate to check that c.f. $(\ref{cfStable})$ satisfies $(\ref{staz})$, for each $\alpha$. Conversely, if $\tilde{\varphi}$ is any c.f. solution of $(\ref{staz})$, taking $\tilde{\varphi}$ as initial datum for $(\ref{eqPhi})$, then $\tilde{\varphi}$ is the solution of the resulting Cauchy problem and, applying Theorem $\ref{teorema1}$, $\tilde{\varphi}$ must be a c.f. like $(\ref{cfStable})$.\qed
\end{proof}
At the best of our knowledge, studies on equation $(\ref{staz})$ hitherto developed assume additional conditions that, in the light of Corollary \ref{corollario2}, appear redundant.

The rest of the present paper is organized in this way. Section 2 includes preliminaries useful to a better understanding of the proof of Theorem $\ref{teorema1}$. Section 3 contains the proof of the theorem, split into four steps.

\section{Preliminaries}
\label{sec:2}

The main feature of our approach is its probabilistic basis inspired by the \textit{Wild representation} of solutions of Boltzmann's-like equations. See \cite{Wild}. In particular, following McKean \cite{McKean}, the Wild representation of the solution of $(\ref{eqPhi})$-$(\ref{condIniz})$ reduces to

\begin{equation}\label{wild}
\vphi(\xi,t)=\sum_{n\geq1}e^{-t}(1-e^{-t})^{n-1}\cq_n(\xi;\vphi_0)\qquad (t\geq0,\xi\in\R)
\end{equation}
with

\begin{equation}\label{cq1}
\cq_1(\cdot,\vphi_0):=\vphi_0(\cdot)
\end{equation}
and, for $n\geq2$,

\begin{equation}\label{cqn}
\cq_n(\cdot;\vphi)=\frac{1}{n-1}\sum_{j=1}^{n-1} \cq_{n-j}(\cdot;\vphi_0)\bullet\cq_j(\cdot;\vphi_0)
\end{equation}
where $\bullet$ denotes an operator called \textit{$p$-Wild convolution}, i.e.

\[
g_1\bullet g_2(\xi):=\frac{1}{2\pi}\int_{0}^{2\pi}g_1(\xi c_p(\theta))g_2(\xi s_p(\theta))d\theta\qquad (\xi\in\R).
\]
Note that the left-hand side of $(\ref{staz})$ is the $p$-Wild convolution of $\vphi$ with iself. So, Corollary $\ref{corollario2}$ states that the class of all fixed points of the $p$-Wild convolution is the same as that of all symmetric stable p.d.'s with index $\alpha=\frac{2}{1+p}$, for each $p\geq0$. It is now important to recall how $(\ref{wild})$-$(\ref{cqn})$ can be used to show that the solution of $(\ref{eqPhi})$-$(\ref{condIniz})$ is the c.f. of a \textit{stochastically weighted sum of real-valued random variables}. For details, cf. \cite{BassettiLadelliMatthes}, \cite{GabettaRegazzini2010} and \cite{McKean}. As it appears in a number of recent papers, this interpretation turns out to be advantageous in the study of convergence to equilibrium of solutions of kinetic equations. It is in fact the key to the applicability of classical powerful methods from the central limit theorem, as already recalled in Section 1. See \cite{BassettiGabetta}, \cite{BassettiLadelli}, \cite{BassettiLadelliMatthes},  \cite{BassettiLadelliRegazzini}, \cite{DoleraGabettaRegazzini}, \cite{DoleraRegazzini2010}, \cite{GabettaRegazzini2006}, \cite{GabettaRegazzini2008}, \cite{GabettaRegazzini2010}, \cite{Regazzini} for one-dimensional models, and \cite{CarlenGabettaRegazzini2007}, \cite{DoleraRegazziniM} for the spatially homogeneous Boltzmann equation.

For each $t>0$, consider a probability space $\Big(\Omega, \CF, \CP_t\Big)$ which supports the following stochastically \textit{independent} random elements:
\begin{itemize}
\item $X=(X_n)_{n\geq1}$ is a sequence of independent and identically distributed (i.i.d.) random elements with common p.d. $\mu_0$ (initial datum).
\item $\tte=(\tte_n)_{n\geq1}$ is a sequence of i.i.d. random numbers uniformly distributed on $(0,2\pi]$.
\item $\ti=(\ti_n)_{n\geq1}$ is a sequence of independent integer-valued random numbers, with $\ti_n$ uniformly distributed on $\{1,\dots,n\}$.
\item $\tnu$ denotes an integer-valued random variable such that $\CP_t\{\tnu=n\}=e^{-t}(1-e^{-t})^{n-1}$ for $n=1,2,\dots\;.$
\end{itemize}

Given these elements, the aforesaid probabilistic interpretation consists in stating that $\cq_n$ is the c.f. of the random number

\begin{equation}\label{Vn}
V^{(n)}:=\sum_{i\geq1}\tbe_{i,n}X_i=(\tbe_n,X)\qquad (n=1,2,\dots)
\end{equation}\newline
where $\tbe_k:=(\tbe_{1,k},\tbe_{2,k},\dots)$ is defined by

\[
\tbe_1:=(1,0,0,\dots)',\qquad \tbe_{k+1}:=A_k(\ti_k)\tbe_k\quad\text{for $k=1,2,\dots$}\;.
\]
$C'$ denotes the transpose of $C$, and $A_k(\ti_k)$ denotes the $\infty^2$ matrix which, for $1<\ti_k<k$, is given by

\[
A_k(\ti_k):=\begin{bmatrix}
						\CI_{\ti_k-1} & O & O & O \\
						O & V_k & O & O \\
						O & O & \CI_{k-\ti_k} & O \\
						O & O & O & O
						\end{bmatrix}
\]
with $\CI_m=$identity matrix, $V_k=(c_p(\tte_k),s_p(\tte_k))'$, while

\[
A_k(\ti_k)=A_k(1):=\begin{bmatrix}
									 V_k & O & O \\
									 O & \CI_{k-1} & O \\
									 O & O & O
									 \end{bmatrix},
									 \qquad
A_k(\ti_k)=A_k(k):=\begin{bmatrix}
									 \CI_{k-1} & O & O \\
									 O & V_k & O \\
									 O & O & O
									 \end{bmatrix}
\]
when $\ti_k=1$ and $\ti_k=k$, respectively. This representation of $\tbe$ is derived from \cite{BassettiLadelliMatthes}. Thanks to $(\ref{Vn})$, one can write

\begin{equation}\label{13}
\begin{split}
\vphi(\xi,t)&=\sum_{n\geq1}\CP_t\{\tnu=n\}\E_t(\exp{\{i\xi(\tbe_n,X)\}}\big|\tnu=n)\\
&=\int_{\Omega}\prod_{k=1}^{\tnu(\omega)}\vphi_0(\tbe_{k,\tnu(\omega)}(\omega)\xi)\CP^{|\CG}_t(d\omega)
\end{split}
\end{equation}
where $\E_t$ denotes the expectation with respect to $\CP_t$, $\CG$ the $\sigma$-algebra generated by $(\tte,\ti,\tnu)$ and $\CP^{|\CG}_t$ the restriction of $\CP_t$ to $\CG$. The latter equality in $(\ref{13})$ can be viewed as a \textit{fibering of the solution of $(\ref{eqPhi})$-$(\ref{condIniz})$ into components which are c.f.'s of weighted sums of i.i.d. random numbers with common p.d. $\mu_0$}.

It is this kind of fibering that paves the way for the study of the problem $(\ref{eqPhi})$-$(\ref{condIniz})$ from the point of view of the central limit problem.

With a view to simplifying computations it is worth recalling that the equality

\begin{equation}\label{req}
\cq_n(\xi;\vphi_0)=\cq_n(\xi;\Re\vphi_0)
\end{equation}
holds for every $n\geq2$ and $\xi$ in $\R$, where $\Re z$ ($\Im z$, respectively) denotes the real (imaginary, respectively) part of the complex number $z$. Then, from $(\ref{wild})$-$(\ref{cqn})$,

\begin{equation}\label{wildRe}
\vphi(\xi,t)=ie^{-t}\Im\vphi_0(\xi)+e^{-t}\sum_{n\geq1}(1-e^{-t})^{n-1}\cq_n(\xi,\Re\vphi_0)
\end{equation}
which is tantamount to saying that the solution to $(\ref{eqPhi})$-$(\ref{condIniz})$ can be viewed as sum of $(ie^{-t}\Im\vphi_0(\xi))$ and the solution to the same Cauchy problem with initial datum $\vphi^*_0:=\Re\vphi_0$. In view of this remark, \textit{we can confine ourselves to considering $(\ref{eqPhi})$-$(\ref{condIniz})$ with a real-valued initial c.f.} $\vphi^*_0$, i.e. with the symmetric p.d. $\mu^*_0$ generated by the p.d. function $F^*_0$ defined in $(\ref{simm})$.

\section{Proof of the theorem}\label{sec:3}
As mentioned in Section 1, the fact that $(\ref{cond})$ is sufficient in order the solution of $(\ref{eqPhi})$-$(\ref{condIniz})$ converge weakly, the limit being a symmetric stable p.d. $\mu_\infty$, is well-known from \cite{BassettiLadelliRegazzini}. Necessity will be proved here by an argument split into a certain number of steps. The basic assumption is that $\mu(\cdot,t)$ converges weakly to some p.d. $\mu_\infty$ as $t$ goes to infinity. The \textit{first step} consists in considering a random array of constituent elements of the central limit problem for sums whose c.f.'s are given by $\xi\mapsto\prod_{k=1}^{\tnu(\omega)}\vphi^*_0(\tbe_{k,\tnu(\omega)}(\omega)\xi)$ for each $\omega$ in $\Omega$ and every $\xi$ in $\R$. In view of the symmetry of $\mu^*_0$, one has $\prod_{k=1}^{\tnu(\omega)}\vphi^*_0(\tbe_{k,\tnu(\omega)}(\omega)\xi)=\prod_{k=1}^{\tnu(\omega)}\vphi^*_0(|\tbe_{k,\tnu(\omega)}(\omega)|\xi)$. The step goes on proving that there is a divergent sequence of instants $t_1<t_2<\dots$ such that the p.d.'s of the above array, under $\CP_{t_1},\CP_{t_2},\dots$ respectively, converge weakly to a p.d.. The argument is based on a technique introduced in \cite{FortiniLadelliRegazzini} and recently utilized in \cite{BassettiLadelliRegazzini}, \cite{DoleraRegazziniM} and \cite{GabettaRegazzini2008} to study convergence to equilibrium in kinetic models. In the \textit{second step}, the aforesaid convergence is combined with the Skorokhod representation theorem to state the existence  of new random arrays which, besides preserving the p.d.'s of the original arrays under the $\CP_{t_n}$'s, converge pointwise. This construction is used to prove the invariability of the limit of specific convergent subsequences of a sequence of the type of $(\ref{Vn})$. The \textit{third step} shows that, in particular, convergent subsequences of $\big(\na\sum_{i=1}^n\zeta_i\big)_{n\geq1}$ belong to the class of the subsequences considered in Step 2, whenever the $\zeta_i$'s are i.i.d. random numbers with common p.d. $\mu^*_0$. In the \textit{fourth step} of the proof, we show that the laws of the random numbers $\na\sum_{i=1}^n\zeta_i$ form a tight class, a fact that, combined with the invariability of the limits of the convergent subsequences, entails convergence of the entire sequence. At this stage, the necessity of $(\ref{cond})$ follows from the central limit theorem for i.i.d. summands.\newline
\newline
\textbf{Step 1} Recall the random elements introduced in Section 2 and, for each $\omega$ in $\Omega$, form the vector
\[
W=W(\omega):=(\tnu(\o),\tte(\o),\ti(\o),\tbe(\o),\tlm(\o),\tLm(\o),\tu(\o))
\]
with:
\begin{itemize}
\item $\tbe=\tbe(\o)=(\tbe_k(\o))_{k\geq1}$.
\item $\tlm=\tlm(\o):=(\tlm_1(\o),\dots,\tlm_{\tnu(\o)}(\o),\delta_0,\delta_0,\dots)$ and, for each $j$ in $\{1,\dots,$ $\tnu(\o)\}$, $\tlm_j(\o)$ is the p.d. determined by the c.f. $\xi\mapsto\vphi^*_0(|\tbe_{j,\tnu(\o)}(\o)|\xi)$, $\xi\in\R$.
\item $\tLm=$convolution of the elements of $\tlm$.
\item $\tu:=(\tu_k)_{k\geq1}$, with $\tu_k=\max_{1\leq j\leq\tnu}\tlm_j([-\frac{1}{k},\frac{1}{k}]^c)$ for every $k\geq1$.
\end{itemize}

The random vector $W(\o)$ is, for each $\o$ in $\O$, the array announced in the above outline of the proof. It contains all the essential elements to characterize the convergence in distribution of the sum $\sum_{j=1}^{\tnu(\o)}\tbe_{j,\tnu(\o)}(\o)\zeta_j$, when $\zeta_1,\zeta_2,...$ are i.i.d. random numbers with common p.d. $\mu^*_0$. The range of $W$ is a subset of
\[
S:=\overline{\N}\times [0,2\pi]^\infty\times \overline{\N}^\infty\times [0,1]^\infty\times (\CP(\overline{\R}))^\infty\times \CP(\overline{\R})\times [0,1]^\infty
\]
where, for any metric space $M$, $\CP(M)$ has to be understood as set of all p.d.'s on $\CB(M)$ endowed with the topology of weak convergence. After metrizing $\CP(\overline{\R})$ consistently with such a topology, $\CP(\overline{\R})$ turns out to be a separable, compact and complete metric space. Therefore, $S$ results in a separable, compact and complete metric space. For an explanation of these facts, see Theorems 6.2, 6.4 and 6.5 in Chapter 2 of \cite{Parthasarathy}. It follows that $\{\CP_tW^{-1}:t\geq0\}$ is a tight family of p.d.'s on $\CB(S)$. Thus, any sequence $\big(\CP_{t_n}W^{-1}\big)_{n\geq1}$, with $t_1<t_2<\dots$ and $t_n\nearrow+\infty$, contains a subsequence $Q_n:=\CP_{t_{m_n}}W^{-1}$, $n=1,2,\dots$, which converges weakly to a probability measure $Q$. At this stage, a straightforward adaptation of Lemma 3 of \cite{GabettaRegazzini2008} implies that $Q$ is supported by
\[
\{+\infty\}\times[0,2\pi]^\infty\times\overline{\N}^\infty\times [0,1]^\infty\times\{\delta_0\}^\infty\times\CP(\R)\times\{0\}^\infty.
\]
Here, we do not reproduce the proof of this fact. Suffice it to recall that the aforesaid lemma relies on the assumption of weak convergence of $\mu(\cdot,t)$ as $t\nearrow+\infty$.\newline
\newline
\textbf{Step 2} Since $S$ is a Polish space, the Skorokhod representation theorem can be applied to state the existence of some probability space  $(\cO,\cF,\cP)$ and of random elements on it, with values in $S$, say\newline
\[
\begin{array}{ll}
&\cW_n:=(\cnun,\cten,\cin,\cben,\clmn,\cLmn,\cun)\\
&\cW:=(+\infty,\ctei,\cii,\cbei,(\delta_0,\delta_0,\dots),\cLmi,(0,0,\dots))
\end{array}
\]
with respective p.d.'s $Q_n$ and $Q$, for $n=1,2,\dots$. Moreover, $\cW_n(\o)\rightarrow\cW(\o)$, with respect to the metric of $S$, for every $\o$ in $\cO$. Hence,

\begin{equation}\label{convergenze}
\cnun\rightarrow+\infty,\quad\cun\rightarrow(0,0,\dots),\quad\clmn_j\Rightarrow\delta_0\quad\text{(for every $j$)},\quad\cLmn\Rightarrow\cLmi
\end{equation}
holds true as $n\rightarrow+\infty$. The distributional properties of $\cW_n$ imply that the equalities $\cLmn=\clmn_1\ast\clmn_2\ast\dots$ and $\cun_k=\max_{1\leq j\leq\cnun}\clmn_j([-\frac{1}{k},\frac{1}{k}]^c)$ hold almost surely. So, with the exception of a set of points $\o$ of $\cP$-probability $0$, the problem of the weak convergence of $\cLmn$ $-$ which is equivalent to the problem of pointwise convergence of $\prod_{k=1}^{\cnun}\vphi^*_0(|\cben_{k,\cnun(\o)}(\o)|\xi)$ $-$ is reduced, in this way, to a central limit problem. See, e.g., Chapter 16 of \cite{FristedtGray}. In particular, in view of $(\ref{convergenze})$, the central limit theorem entails the existence of a L\'evy measure $\mu$, which depends on $\o$, such that

\begin{equation}\label{levy}
\mu([x,+\infty))=\mu((-\infty,-x])=\lim_{n\rightarrow+\infty}\sum_{j=1}^{\cnun}\clmn_j([x,+\infty))
\end{equation}
holds for every $x>0$. Moreover, since there is a set of $\cP$-probability $1$ on which $|\cben_j|$ are strictly positive, then

\[
\clmn_j([x,+\infty))=1-F^*_0\Big(\frac{x}{|\cben_{j,\cnun}|}\Big)\qquad (x>0)
\]
on such set. It is now an easy task to show that $\liminf_{x\rightarrow+\infty}x^\alpha(1-F^*_0(x))<+\infty$ as already recalled apropos of Theorem 2.1 in \cite{BassettiLadelliRegazzini}. This fact, which prevents the function $\rho\colon(0,+\infty)\rightarrow\R^+$ defined by

\[
\rho(x):=x^\alpha(1-F^*_0(x))
\]
from converging to $+\infty$ as $x\rightarrow+\infty$, has a further deeper consequence, i.e.

\begin{equation}\label{limsup}
\limsup_{x\rightarrow+\infty}\rho(x)<+\infty.
\end{equation}
This is tantamount to saying that there are strictly positive constants $c$ and $A$ for which

\begin{equation}\label{condlimsup}
1-F_0(x)\leq\frac{c}{x^\alpha}\quad\text{and}\quad F_0(-x)\leq\frac{c}{x^\alpha}
\end{equation}
hold for every $x\geq A$. Indeed, if $(\ref{limsup})$ or, equivalently, $(\ref{condlimsup})$ does not hold true, then there is a strictly increasing and positive sequence $(x_n)_{n\geq1}$ such that $x_n\nearrow+\infty$ and $\rho(x_n)\rightarrow+\infty$, as $n\rightarrow+\infty$. Hence, for every subsequence $(x_{n_k})_{k\geq1}$ and sequence $(\beta_k)_{k\geq1}$ in $(0,1)$, the combinations $\overline{x}_{\beta_k}:=\beta_k x_{n_k}+(1-\beta_k)x_{n_{k+1}}$ satisfy

\[
1-F^*_0(x_{n_k})\geq 1-F^*_0(\overline{x}_{\beta_k})\geq 1-F^*_0(x_{n_{k+1}})
\]
and this entails

\[
\rho(\overline{x}_{\beta_k})\geq(1-\beta_k)^\alpha\rho(x_{n_{k+1}})
\]
which in turn implies that $\lim_{k\rightarrow+\infty}\rho(\overline{x}_{\beta_k})=+\infty$ whenever $\liminf_{k\rightarrow+\infty}(1-\beta_k)>0$. This conclusion, in view of the arbitrariness of the choice of $(x_{n_k})_{k\geq1}$ and $(\beta_k)_{k\geq1}$, is in contradiction with the existence of a finite lower limit of $\rho(x)$ as $x\rightarrow+\infty$.

We continue with the second step by presenting complete definitions of $\cO,\cF$ and of the random elements $\cW_n$ and $\cW$. This presentation is useful since, on the one hand, it will be actually used in the sequel and, on the other hand, it slightly deviates from the standard exposition provided, for example, in \cite{Billingsley} or \cite{Dudley}. Motivation for this deviation will become transparent at the end of this step. We start the exposition by noting that, in view of the separability of $S$, for every $m$ and $\veps_m=\frac{1}{2^{m+1}}$ there is a partition $\{B^m_0,B^m_1,\dots,B^m_{k_m}\}$ of $S$ satisfying the conditions:

\begin{itemize}
\item [] $B^m_1,\dots,B^m_{k_m}$ are open;
\item [] $Q(B^m_0)<\veps_m$;
\item [] $Q(\partial B^m_i)=0\quad(i=0,\dots,k_m),\quad diamB^m_i<\veps_m\quad(i=1,\dots,k_m)$.
\end{itemize}
Partitions are defined in such a way that the ($m+1$)-th partition is a refinement of the $m$-th. Moreover, all the $B^m_i$ for which $Q_n(B^m_i)=0$ infinitely often [and, then, $Q(B^m_i)=0$] will be amalgamated with $B^m_0$. [In the standard presentations, \textit{all} the $B^m_i$ for which $Q(B^m_i)=0$ are amalgamated with $B^m_0$.] Choose the strictly positive integer $n_m$ in such a way that

\begin{equation}
\begin{array}{ll}
& Q_n(B^m_0)\geq (1-\veps_m)Q(B^m_0)\\
& Q_n(B^m_i)> (1-\veps_m)Q(B^m_i)
\end{array}
\end{equation}
hold for every $i=1,\dots,k_m$ and every $n\geq n_m$ for $m=1,2,\dots$. We can assume $n_1<n_2<\dots$. Now, we form the product space

\[
\cO:=S_1\times S_2\times S_2\times\dots\times\CM_1\times\CM_2\times\dots\times S_3\times S_3\times\dots\times[0,1]
\]
where:

\begin{itemize}
\item $S_1=S_2=S_3=S$,
\item For each $j$, $\CM_j$ stands for the set of all $(n_{j+1}-n_j)\times(k_j+1)$ matrices, whose elements take values in suitable Borel subsets of $S$ according to the forthcoming descriptions of the $Q_n(\cdot|B^j_i)$'s and of the $\cY^j_{n+n_j-1,i-1}$'s.
\end{itemize}
Let

\[
\cW,\cY_1,\dots,\cY_{n_1-1},\cM_1,\cM_2,\dots,\cZ_1,\cZ_2,\dots,\cz
\]
be the coordinate variables of $\cO$, and let $\cP$ be a probability measure on $\CB(\cO)$ which makes these coordinates stochastically independent with marginal laws ($\CL(\cdot)$, for short) satisfying

\begin{equation*}
\begin{array}{ll}
& \CL(\cW)=Q,\qquad\CL(\cY_j)=Q_j\quad(j=1,\dots,n_1-1),\\
& \CL(\cM_j)=\otimes_{n=n_j}^{n_{j+1}-1}\otimes_{i=0}^{k_j}Q_n(\cdot\big|B^j_i)\quad(j=1,2,\dots),\\
& \CL(\cz)=\text{uniform p.d. on $[0,1]$.}
\end{array}
\end{equation*}
Each $Q_n(\cdot\big|B^j_i)$ must be understood as a p.d. on the restriction of $\CB(S)$ to $B^j_i$. It is uniquely specified for every $i=1,\dots,k_j$ whenever $n\geq n_j$. For $i=0$, $Q_n(B^j_0)$ could be $0$ even if $n\geq n_j$; in such case define $Q_n(\cdot\big|B^j_0)$ to be any p.d. on $(S,\CB(S))$. We are now in a position to define $\CL(\cZ_n)$ as

\[
\CL(\cZ_n)=\mu_n(\cdot):=\frac{1}{\veps_m}\sum_{i=0}^{k_m}Q_n(\cdot\big|B^m_i)[Q_n(B^m_i)-(1-\veps_m)Q(B^m_i)].
\]
As to the matrix $\cM_j$, denote its ($n,i$)-th entry by $\cY^j_{n+n_j-1,i-1}$ with $j=1,2,\dots$, $i=1,\dots,k_j+1$, $n=1,\dots,n_{j+1}-n_j$ and assume that the range of $\cY^j_{n+n_j-1,i-1}$ is $B^j_i$. From the definition of $\CL(\cM_j)$, one has $\CL(\cY^j_{n,i})=Q_n(\cdot\big|B^j_i)$ for $j\in\N$, $n_j\leq n\leq n_{j+1}-1$, $i=0,\dots,k_j$. So, if one introduces the random elements

\[
\cW_n:=\cY_n\qquad\text{for $n<n_1$}
\]
and, for $n_m\leq n<n_{m+1}$ and $m=1,2,\dots$,

\[
\cW_n:=\sum_{i=0}^{k_m}\I_{\{\cz\leq1-\veps_m, \cW\in B^m_i\}}\cY^m_{n,i}+\I_{\{\cz>1-\veps_m\}}\cZ_n,
\]
it is easy to show that $\CL(\cW_n)=Q_n$ for every $n$. Note that $\I_A$ denotes the indicator of the set $A$. Moreover, as to the event $E:=\bigcup_{M\geq1}\bigcap_{m\geq M}E_m$ with

\[
E_m:=\{\cW\notin B^m_0,\cz\leq 1-\veps_m\}\qquad(m=1,2,\dots)
\]
one gets $\cP(E)=1$ and, for each of its points $\co$, there is $\mbar=\mbar(\co)$ such that, for $m\geq\mbar$,

\begin{itemize}
\item [(i)] $\cW(\co)\notin B^m_0$;
\item [(ii)] $\cz(\co)\leq1-\veps_m$;
\item [(iii)] $dist(\cW_n(\co),\cW(\co))\leq\veps_m$ for $n=n_m,\dots,n_{m+1}-1$.
\end{itemize}
It follows of course that $dist(\cW_n,\cW)\rightarrow0$ almost surely ($\cP$).

As an application of the above construction we discuss the behaviour of particular subsequences of $(\cW_n)_{n\geq1}$. Let $((b_{l,r'})^{r'}_{l=1})_{r'}$ and $((b_{l,r''})^{r''}_{l=1})_{r''}$ be subsequences of a sequence $((b_{l,r})^{r}_{l=1})_{r}$ satisfying:

\begin{itemize}
\item [(I)] For every $r'$ and $r''$, $(b_{l,r'})^{r'}_{l=1}$ and $(b_{l,r''})^{r''}_{l=1}$ belong to the support of $(\cbe_{k,\cnu})^{\cnu}_{k=1}$.
\item [(II)] $\sum_{l=1}^{r'}b_{l,r'}\zeta_l$ and $\sum_{l=1}^{r''}b_{l,r''}\zeta_l$ converge in law as $r'$ and $r''$ go to infinity, when the $\zeta_l$'s are i.i.d. random numbers with common p.d. $\mu^*_0$; let $\lm'$ and $\lm''$ denote the respective limits.
\end{itemize}
Define $\cW'$ and $\cW''$ to be realizations of $\cW$ with $\lm'$ and $\lm''$ as penultimate coordinate, respectively. In view of (I) combined with the indicated changes in the Skorokhod construction, partitions can be defined so that both $\cW'$ and $\cW''$ belong to $(B^m_0)^c$ for every $m$. Then, if $\lm'\neq\lm''$, $\cW'$ and $\cW''$ belong to distinct elements of the $m$-th partition for all sufficiently large $m$; for small values of $\eta>0$, the balls $B(\cW',\eta)$ and $B(\cW'',\eta)$ with radius $\eta$ and centers $\cW'$, $\cW''$, respectively, turn out to be disjoint. Now, for any $\cw$ which belongs to $(B^m_0)^c$ for every $m$, consider the following closed subset of $\cO$

\begin{eqnarray*}
F_\eta:=\Big(\overline{B(\cw,\eta)}\Big)&\times&\Big(\overline{B(\cW',\eta)}\Big)^{n_1-1}\times\Big(\overline{B(\cW'',\eta)}\Big)^{(n_2-n_1)(1+k_1)}\times\\
&\times&\Big(\overline{B(\cW',\eta)}\Big)^{(n_3-n_2)(1+k_2)}\times\dots\times[0,\eta]
\end{eqnarray*}
where $\overline{A}$ denotes the closure of the set $A$. Given any strictly positive sequence $(\eta_q)_{q\geq1}$ such that $\eta_q\searrow0$, $\bigcap_{q\geq1}F_{\eta_q}$ contains exactly one point (from the completeness of $\cO$ and the Cantor intersection theorem), say $\co$. This point belongs to $E$,

\[
\cW_n(\co)=\sum_{i=1}^{k_m}\I_{\{\cW\in B^m_i\}}\cY^m_{n,i}(\co)\qquad(n\geq n_m, m\geq M(\co)).
\]
Hence, from (iii), $\cW_n(\co)$ belongs to an arbitrary neighborhood of $\cW(\co)$ except for a finite set of $n$. On the other hand, the definition of $\co$ combined with the above representation of $\cW_n(\co)$, leads to state the existence of two subsequences of $(\cW_n(\co))_{n\geq1}$ which, except for a finite set of $n$, belong to two disjoint neighborhoods of $\cW'$ and $\cW''$, respectively. Since this contradicts the convergence of $(\cW_n(\co))_{n\geq1}$, one concludes that $\lm'=\lm''$.\newline
\newline
\textbf{Step 3} In this step we exhibit a remarkable example of subsequences satisfying conditions (I)-(II) in Step 2, that will be used to prove the necessity of $(\ref{cond})$. Consider the sequence $\big(\na\sum_{l=1}^n\zeta_l\big)_{n\geq1}$ where the $\zeta_l$'s are the same as in Step 2. We concentrate our attention on the subsequence $\big(\frac{1}{(2n)^\frac{1}{\alpha}}\sum_{l=1}^{2n}\zeta_l\big)_{n\geq1}$ since, on the one hand, it converges if and only if the entire sequence converges and, on the other hand, it allows a simpler check of the above condition (I). Assume that $(r')$ and $(r'')$ are subsequences of $(2n)_{n\geq1}$ such that $\big(\frac{1}{(r')^\frac{1}{\alpha}}\sum_{l=1}^{r'}\zeta_l\big)_{r'}$ and $\big(\frac{1}{(r'')^\frac{1}{\alpha}}\sum_{l=1}^{r''}\zeta_l\big)_{r''}$ converge in distribution. Here, we show that these subsequences satisfy condition (I). Too see this, consider any $r'=2^m+2k$ and $(\ci_1,\dots,\ci_{r'-1})$ for $m$ in $\{1,2,\dots\}$ and $k$ in $\{0,\dots,2^{m-1}-1\}$. Note that with each realization of $(\ci_1,\dots,\ci_{r'-1})$ one can associate a binary (McKean) tree with $2^m+2k$ leaves. Leaves are ordered  from left to right and labelled by integers $1,2,\dots,2^m+2k$. Now, contemplate only the symmetric tree $\textbf{t}_s$ obtained from the complete tree $\textbf{t}_c$ with $2^m$ leaves $-$ in which all the leaves have the same \textit{depth} $m$ $-$ by attaching $2k$ copies of the unique tree with two leaves to the leaves of $\textbf{t}_c$ numbered by $1,2^m,2,2^m-1,\dots,k,2^m-(k-1)$. The probability of obtaining such a tree $\textbf{t}_s$ with $\cnu=r'$ is plainly strictly positive, and it is also easy to check that the support of the p.d. of the corresponding vectors $(\cbe_{1,r'},\dots,\cbe_{r',r'})$ includes $(\frac{1}{(r')^\frac{1}{\alpha}},\dots,\frac{1}{(r')^\frac{1}{\alpha}})$. This last statement follows from the continuity of the laws of $c_p(\cte_i)$ and $s_p(\cte_i)$ combined with the independence of the $\cte_i$'s. An analogous conclusion holds of course for $r''$. Then, condition (I) in Step 2 is verified by each subsequence $\big(\frac{1}{r^\frac{1}{\alpha}}\sum_{l=1}^{r}\zeta_l\big)_{r}$ when $r=r',r''$. To get a complete connection between these subsequences and those considered at the end of Step 2, it remains to prove that $(r')$ and $(r'')$ can be chosen in such a way that the corresponding subsequences of $\big(\frac{1}{(2n)^\frac{1}{\alpha}}\sum_{l=1}^{2n}\zeta_l\big)_{n\geq1}$ converge in distribution. We shall prove this fact by showing that the laws of the sums $\na\sum_{l=1}^n\zeta_l$ form a tight family. At that stage, Step 2 can be applied to obtain invariability of the limits which, combined with tightness, yields convergence in law of the entire sequence $\big(\frac{1}{(2n)^\frac{1}{\alpha}}\sum_{l=1}^{2n}\zeta_l\big)_{n\geq1}$.\newline
\newline
\textbf{Step 4} The proof of tightness relies on the well-known inequality

\begin{equation}\label{doob}
\CP_t\{|S_{\alpha,n}|>C\}\leq\frac{1}{\Delta}\Big(1+\frac{2\pi}{C\Delta}\Big)^2\int_0^\Delta\Big[1-\vphi^*_0\Big(\frac{\xi}{n^\frac{1}{\alpha}}\Big)^n\Big]d\xi
\end{equation}
where

\[
S_{\alpha,n}:=\na\sum_{i=1}^nX_i\qquad(n=1,2,\dots)
\]
and the $X_i$'s, according to Section 2, are i.i.d. random numbers under $\CP_t$, with common p.d. $\mu^*_0$. See Section 2 for symbols and Subsection 8.3 of \cite{ChowTeicher} for the Doob inequality $(\ref{doob})$. Now, if

\begin{equation}\label{limiti}
0\leq 1-\vphi^*_0\Big(\frac{\xi}{n^\frac{1}{\alpha}}\Big)\leq\frac{3}{8},
\end{equation}
then there is $\theta$ such that $0\leq|\theta|\leq1$ and

\[
n\log\vphi^*_0\Big(\frac{\xi}{n^\frac{1}{\alpha}}\Big)=-n\Big\{1-\vphi^*_0\Big(\frac{\xi}{n^\frac{1}{\alpha}}\Big)\Big\}+\frac{4\theta}{5}n\Big\{1-\vphi^*_0\Big(\frac{\xi}{n^\frac{1}{\alpha}}\Big)\Big\}^2.
\]
On the other hand, there is $\Delta_0>0$ so that $(\ref{limiti})$ holds true if $|\xi|\leq\Delta n^\frac{1}{\alpha}$ for any $\Delta$ in $(0,\Delta_0]$ and, therefore,

\begin{eqnarray*}
1-\vphi^*_0\Big(\frac{\xi}{n^\frac{1}{\alpha}}\Big)^n&=& 1-\exp\Big\{-n\Big[1-\vphi^*_0\Big(\frac{\xi}{n^\frac{1}{\alpha}}\Big)\Big]+\frac{4\theta}{5}n\Big[1-\vphi^*_0\Big(\frac{\xi}{n^\frac{1}{\alpha}}\Big)\Big]^2\Big\}\\
&\leq& 1-\exp\Big\{-n\Big[1-\vphi^*_0\Big(\frac{\xi}{n^\frac{1}{\alpha}}\Big)\Big]-\frac{3}{10}n\Big[1-\vphi^*_0\Big(\frac{\xi}{n^\frac{1}{\alpha}}\Big)\Big]\Big\}\\
&\leq& 1-\exp\Big\{-\frac{13}{10}n\Big[1-\vphi^*_0\Big(\frac{\xi}{n^\frac{1}{\alpha}}\Big)\Big]\Big\}\\
&\leq& \frac{13}{10}n\Big[1-\vphi^*_0\Big(\frac{\xi}{n^\frac{1}{\alpha}}\Big)\Big].
\end{eqnarray*}
Hence, from $(\ref{doob})$,

\[
\CP_t\{|S_{\alpha,n}|>C\}\leq\frac{1}{\Delta}\Big(1+\frac{2\pi}{C\Delta}\Big)^2\frac{13}{10}n\int_0^\Delta\Big[1-\vphi^*_0\Big(\frac{\xi}{n^\frac{1}{\alpha}}\Big)\Big]d\xi.
\]
As to the last integral,

\begin{eqnarray*}
\int_0^\Delta\Big[1-\vphi^*_0\Big(\frac{\xi}{n^\frac{1}{\alpha}}\Big)\Big]d\xi &=& 2\int_0^\Delta\Big(\int_0^{+\infty}\Big[1-\cos\frac{\xi x}{n^\frac{1}{\alpha}}\Big]dF^*_0(x)\Big)d\xi\\
&=& 2\lim_{\O\rightarrow+\infty}\int_{0}^{\O}\Big[\Delta-\frac{n^\frac{1}{\alpha}}{x}\sin\frac{\Delta x}{n^\frac{1}{\alpha}}\Big]dF^*_0(x)\\
&=& 2 n^\frac{1}{\alpha}\int_0^{+\infty}[1-F^*_0(x)]\frac{1}{x^2}\Big[\sin\frac{\Delta x}{n^\frac{1}{\alpha}}-\frac{\Delta x}{n^\frac{1}{\alpha}}\cos\frac{\Delta x}{n^\frac{1}{\alpha}}\Big]dx.
\end{eqnarray*}
Hence, putting

\[
\tau(x):=\sin x-x\cos x\qquad(x>0)
\]
and recalling the definition of $\rho$, one has

\begin{eqnarray*}
&&2 n^\frac{1}{\alpha}\int_0^{+\infty}[1-F^*_0(x)]\frac{1}{x^2}\Big[\sin\frac{\Delta x}{n^\frac{1}{\alpha}}-\frac{\Delta x}{n^\frac{1}{\alpha}}\cos\frac{\Delta x}{n^\frac{1}{\alpha}}\Big]dx=\\
&&\qquad\qquad\qquad\qquad\qquad\qquad=2\Delta^{1+\alpha}\frac{1}{n}\int_{0}^{+\infty}\frac{1}{y^{2+\alpha}}\rho\Big(\frac{yn^\frac{1}{\alpha}}{\Delta}\Big)\tau(y)dy.
\end{eqnarray*}
Then, setting $M:=\sup_{x\geq0}\rho(x)$,

\begin{eqnarray*}
\CP_t\{|S_{\alpha,n}|>C\}&\leq&\frac{13}{5}\Big(1+\frac{2\pi}{C\Delta}\Big)^2M\Delta^\alpha\int_{0}^{+\infty}\frac{1}{y^{2+\alpha}}\Big[\frac{y^3}{6}\I_{(0,\veps)}(y)+\\
&&\qquad\qquad\qquad\qquad\qquad+(1+y)\I_{(\veps,+\infty)}(y)\Big]dy\\
&\leq& A\Delta^\alpha\Big(1+\frac{2\pi}{C\Delta}\Big)^2
\end{eqnarray*}
obtains for some suitable constant $A$. So, taking $C=\Delta^{-1}\geq\Big(\frac{A(1+2\pi)^2}{\veps}\Big)^\frac{1}{\alpha}$, one gets

\[
\CP_t\{|S_{\alpha,n}|>C\}\leq\veps\qquad(t>0)
\]
that entails the tightness of the laws of the sums $\na\sum_{i=1}^nX_i$. Then, in view of the final remark in Step 3, we conclude that these sums converge in distribution. This implies $-$ through the fundamental theorem on the standard domain of attraction (see, e.g. the final part of Section 2.6 of \cite{IbragimovLinnik}) $-$ that $F^*_0$ must satisfy $(\ref{cond})$. Hence, from Theorem 2.1 in \cite{BassettiLadelliRegazzini},

\[
\lim_{t\rightarrow+\infty}\vphi(\xi,t)=\lim_{t\rightarrow+\infty}\vphi^*(\xi,t)=\exp\{-a_0|\xi|^\alpha\}
\]
for every $\xi$. This terminates the proof.


%
%

\begin{acknowledgements}
We would like to thank Eleonora Perversi for very helpful comments. The research of Ester Gabetta has been partially supported by MIUR-2009 NAPJF-002, that of Eugenio Regazzini by MIUR-2008 MK3AFZ.
\end{acknowledgements}



\end{document}